\documentclass[12pt]{amsart}
\usepackage{amssymb}

\newtheorem{theorem}{Theorem}[section]

\newtheorem{corollary}[theorem]{Corollary}

\theoremstyle{definition}

\theoremstyle{remark}
\newtheorem{remark}[theorem]{Remark}

\numberwithin{equation}{section}

\newcommand{\g}[2]{\mbox{$\langle #1 ,#2 \rangle$}}
\newcommand{\s}[1]{\mbox{$\mathbb{S}^{#1}$}}
\newcommand{\h}[1]{\mbox{$\mathbb{H}^{#1}$}}
\newcommand{\R}[1]{\mbox{$\mathbb{R}^{#1}$}}
\newcommand{\Z}{\mbox{$\mathbb{Z}$}}
\newcommand{\m}{\mbox{$\Sigma$}}
\newcommand{\n}{\mbox{$M$}}
\newcommand{\nablabar}{\mbox{$\overline{\nabla}$}}

\newcommand{\fm}{\mbox{$\mathcal{C}^\infty(\m)$}}

\newcommand{\x}{\mbox{$\psi:\m^2\fle\n^3$}}
\newcommand{\xhom}{\mbox{$\psi:\m^2\fle\E3$}}
\newcommand{\fle}{\mbox{$\rightarrow$}}
\newcommand{\rf}[1]{\mbox{(\ref{#1})}}

\def\area{\mathop\mathrm{Area(\m)}\nolimits}

\def\tr{\mathop\mathrm{tr}\nolimits}
\def\Kbar{\mathop{\overline{K}}\nolimits}
\def\KbarSigma{\mathop{\overline{K}_\Sigma}\nolimits}
\def\RiemE{\mathop\mathrm{\overline{R}}\nolimits}
\def\Ric{\mathop\mathrm{\overline{Ric}}\nolimits}
\def\RicN{\mathop\mathrm{\overline{Ric}(\textit{N},\textit{N})}\nolimits}

\def\xfrak{\mathfrak{X}}
\def\xfrakM{\xfrak(\m)}

\def\E3{\mbox{$\mathbb{E}^3(\kappa,\tau)$}}
\def\sB{\mbox{$\mathbb{S}^3_b(\kappa,\tau)$}}
\def\nil{\mbox{$Nil_3(\tau)$}}
\def\univC{\mbox{$\widetilde{Sl(2,\R{})}(\kappa,\tau)$}}
\def\sL{\mbox{$Sl(2,\R{})(\kappa,\tau)$}}
\def\integ{\int_\Sigma}

\begin{document}

\title[On the first stability eigenvalue of CMC surfaces]
{On the first stability eigenvalue of constant mean curvature surfaces into homogeneous 3-manifolds}

\author{Luis J. Al\'ias}
\address{Departamento de Matem\'{a}ticas, Universidad de Murcia, E-30100 Espinardo, Murcia, Spain}
\email{ljalias@um.es}
\thanks{This work was partially supported by MECD (Ministerio de Educaci\'on, Cultura y Deporte) Grant  no FPU12/02252, MINECO (Ministerio de Econom\'{\i}a y Competitividad) and FEDER (Fondo Europeo de Desarrollo Regional) project MTM2012-34037 and Fundaci\'{o}n S\'{e}neca project 04540/GERM/06, Spain. This research is a result of the activity developed within the framework of the Programme in Support of Excellence Groups of the Regi\'{o}n de Murcia, Spain, by Fundaci\'{o}n S\'{e}neca, Regional Agency for Science and Technology (Regional Plan for Science and Technology 2007-2010).}

\author{Miguel A. Mero\~no}
\address{Departamento de Matem\'{a}ticas, Universidad de Murcia, E-30100 Espinardo, Murcia, Spain}
\email{mamb@um.es}

\author{Irene Ortiz}
\address{Departamento de Matem\'{a}ticas, Universidad de Murcia, E-30100 Espinardo, Murcia, Spain}
\email{irene.ortiz@um.es}

\subjclass[2010]{53C42, 53C30}

\date{\today}



\begin{abstract}
We find out upper bounds for the first eigenvalue of the stability operator for compact constant
mean curvature surfaces immersed into certain 3-dimensional Riemannian spaces, in particular into homogeneous
3-manifolds. As an application we derive some consequences for strongly stable surfaces in such ambient spaces.
Moreover, we also get a characterization of Hopf tori in certain Berger spheres.
\end{abstract}

\maketitle

\section{Introduction}
\label{s1}

Let \x\ be a compact surface immersed into a 3-dimensional Riemannian space. We will assume that
\m\ is \textit{two-sided}, which means that there exists a unit normal vector field $N$ globally
defined on $\m$, and will denote by $A$ its second fundamental form (with respect to $N$) and by
$H$ its mean curvature, $H=(1/2)\tr(A)$. Every smooth function $f\in\fm$ induces a normal variation
$\psi_t$ of the immersion $\psi$, with variational normal field $fN$ and first variation of the
area functional $\mathcal{A}(t)$ given by
\[
\delta_f\mathcal{A}=\frac{d\mathcal{A}}{dt}(0)=-2\integ fH.
\]
As a consequence, minimal surfaces ($H=0$) are characterized as critical points of the area
functional whereas constant mean curvature (CMC) surfaces are critical points of the area functional
restricted to smooth functions $f$ which satisfy the additional condition $\integ f=0$.
Geometrically, such additional condition means that the va\-riations under consideration preserve a
certain volume function.

For such critical points, the stability of the corresponding variational problem is given by the
second variation of the area functional,
\[
\delta^2_f\mathcal{A}=\frac{d^2\mathcal{A}}{dt^2}(0)=-\integ  fJf,
\]
with $Jf=\Delta f+\left(|A|^2+\RicN\right)f$, where $\Delta$ stands for the Laplacian operator on
\m\ and $\Ric$ denotes the Ricci curvature of $\n^3$. The surface $\m$ is said to be strongly
stable if $\delta^2_f\mathcal{A}\geq 0$, for every $f\in\fm$. The operator $J=\Delta+|A|^2+\RicN$
is called the Jacobi or stability operator of the surface, and it is a Schr\"odinger operator. As
is well known, the spectrum of $J$
\[
\mathrm{Spec}(J)=\{ \lambda_1<\lambda_2<\lambda_3<\cdots \}
\]
consists of an increasing sequence of eigenvalues $\lambda_k$ with finite multiplicities $m_k$ and
such that $\lim_{k\rightarrow\infty}\lambda_k=+\infty$. Moreover, the first eigenvalue is simple
($m_1=1$) and it satisfies the following min-max characterization
\begin{equation}
\label{minmax} \lambda_1=\min \left\{\frac{-\integ  fJ(f)}{\integ  f^2}: \quad f\in\fm,f\neq0
\right\}.
\end{equation}
In terms of the spectrum, $\m$ is strongly stable if and only if $\lambda_1\geq 0$.

Observe that with our criterion, a real number $\lambda$ is an eigenvalue of $J$ if and only if
$J(f)+\lambda f=0$ for some smooth function $f\in\fm$, $f\neq0$.

In 1968, Simons \cite{Sim} found out an estimate for the first eigenvalue of $J$ on any compact
minimal hypersurface in the standard sphere. In particular, for minimal surfaces in the 3-sphere he
proved that $\lambda_1=-2$ if the surface is totally geodesic and $\lambda_1\leq -4$ otherwise.
Later on, Wu \cite{Wu} characterized the equality by showing that it holds only for the minimal
Clifford torus. More recently, Perdomo \cite{Pe} gave a new proof of this spectral characterization
by getting an interesting formula that relates the first eigenvalue $\lambda_1$, the genus of the
surface, the area and a simple invariant. Finally,  Al\'ias, Barros and Brasil \cite{ABBr} extended
Wu and Perdomo's results to the case of CMC hypersurfaces in the standard
sphere, characterizing some CMC Clifford tori.

The standard 3-sphere is a simply connected space form (the compact one). Besides them, the most regular Riemannian 3-manifolds are the homogeneous ones and between them, the only compact are the Berger spheres and their quotients. In
the last years, the CMC surfaces of the homogeneous Riemannian 3-manifolds have been deeply studied (the starting point was the work of Abresch and Rosenberg \cite{AR}) and several stability results have been proved (see
\cite{ToUrStability} and \cite{S}).

In this paper, we look for estimates for $\lambda_1$ and some characterizations for compact
CMC surfaces into 3-dimensional Riemannian spaces with sectional curvature
bounded from below. In particular, we find out upper bounds for $\lambda_1$ for compact constant
mean curvature surfaces into homogeneous 3-manifolds.
As an application we derive some consequences for strongly stable surfaces in such ambient spaces.
Moreover, we get also a characterization of Hopf tori in certain Berger spheres and into the product $\s{2}\times\s{1}$.

\section{First results}

Let \x\ be a compact two-sided surface with constant mean curvature $H$ immersed into a
3-dimensional Riemannian space. Let us introduce the so called traceless second fundamental form of
$\m$, that is, the tensor $\phi$ given by $\phi=A-HI$, where $I$ denotes the identity operator on
$\xfrakM$. Observe that $\tr(\phi)=0$ and $|\phi|^2=|A|^2-2H^2\geq0$, with equality if and only if
\m\ is totally umbilical. For that reason $\phi$ is also called the total umbilicity tensor of
$\m$. In terms of $\phi$, the Jacobi operator is given by
\[
J=\Delta+|\phi|^2+2H^2+\RicN.
\]

Now let us choose a first positive eigenfunction $\rho\in\fm$
of the stability operator. Thus $J\rho=-\lambda_1\rho$ or, equivalently,
\begin{equation}
\label{deltarho} \Delta\rho=-\left(\lambda_1+|A|^2+\RicN\right)\rho.
\end{equation}
Extending Perdomo's ideas \cite[Section 3]{Pe} to our more general case, one can compute
\[
\Delta
\textrm{log}\rho=\rho^{-1}\Delta\rho-\rho^{-2}|\nabla\rho|^2=-\left(\lambda_1+|A|^2+\RicN\right)-\rho^{-2}|\nabla\rho|^2,
\]
and integrate on \m\ to find
\[
\alpha=\integ\rho^{-2}|\nabla\rho|^2=-\lambda_1\area-\integ\left(|A|^2+\RicN\right),
\]
where $\alpha\geq0$ defines a simple invariant that is independent of the choice of $\rho$ because
$\lambda_1$ is simple. In other words
\[
\lambda_1=-\frac{1}{\area}\left(\alpha+\integ\left(|A|^2+\RicN\right)\right).
\]
Now from the Gauss equation, we obtain a relation between the norm of the shape operator $|A|^2$,
the sectional curvature $\KbarSigma$ of the tangent plane to $\m$ in $\n^3$, and the Gaussian
curvature $K$ of the surface as $|A|^2=2(2H^2+\KbarSigma-K)$ and, by the Gauss-Bonnet Theorem, the
above formula becomes
\begin{equation}
\label{lambda1}
\lambda_1=-4H^2-\frac{1}{\area}\left(\alpha+8\pi(g-1)+\integ\left(2\KbarSigma+\RicN\right)\right).
\end{equation}

As a first consequence, we can establish the following result.

\begin{theorem}
\label{th1.1} Let $\n^3$ be a 3-dimensional Riemannian space with sectional curvature $\Kbar$
bounded from below by a constant c. Consider $\m^2$ a compact two-sided surface immersed into
$\n^3$ with constant mean curvature $H$, and let $\lambda_1$ stand for the first eigenvalue of its
Jacobi operator. Then
\begin{itemize}
\item[(i)] $\lambda_1\leq-2(H^2+c)$, with equality if and only if \m\ is totally umbilic in
$\n^3$ and the normal direction to \m\ is a direction of minimum Ricci curvature of $\n^3$ equals $2c$,
\item[(ii)] $\lambda_1\leq -4(H^2+c)-8\pi(g-1)/\area$, where g denotes the genus of $\m$.
Moreover, equality holds if and only if \m\ has constant Gaussian curvature, $\KbarSigma=c$ and the
normal direction to \m\ is a direction of minimum Ricci curvature of $\n^3$ equals $2c$.
\end{itemize}
\end{theorem}

\begin{proof}
Taking the constant function $f=1$ as a test function in \rf{minmax} to estimate $\lambda_1$, one
easily gets that
\begin{eqnarray}
\label{lambda1.1}
\nonumber \lambda_1&\leq&\frac{-\integ 1J1}{\integ
1^2}=-2H^2-\frac{1}{\area}\integ\RicN-\frac{1}{\area}\integ|\phi|^2\\
&\leq& -2(H^2+c)-\frac{1}{\area}\integ|\phi|^2\leq-2(H^2+c)
\end{eqnarray}
Moreover, if the equality holds then $\phi=0$ on $\m$ and  $\RicN=2c$. So $\m$ is totally umbilic and because of $2c\leq\Ric(X,X)$ for all
$X\in\xfrak(\n^3)$, the normal direction to \m\ is a direction of minimum Ricci curvature of
$\n^3$. Conversely, if $\m$ is totally umbilic and $\RicN=2c$ then $J=\Delta+2H^2+2c$ and so $\lambda_1=-2H^2-2c$.

On the other hand, from \rf{lambda1} we get that $\lambda_1\leq -4(H^2+c)-8\pi(g-1)/\area$ since
$2\KbarSigma+\RicN\geq 4c$ and $\alpha\geq 0$, which proves the first statement of part (ii).
Moreover, if the equality holds then $\alpha=0$, $\KbarSigma=c$ and $\RicN=2c$, so the normal
direction to \m\ is a direction of minimum Ricci curvature of $\n^3$ as above. The fact that
$\alpha=0$ implies $\rho$ is constant and from \rf{deltarho} we see that $|A|^2$ is also constant,
which means that $K$ is constant by the Gauss equation. Conversely, under such hypothesis we have
$J=\Delta+4(H^2+c)-2K$, hence $\lambda_1=-4(H^2+c)+2K=-4(H^2+c)-8\pi(g-1)/\area$ by the
Gauss-Bonnet formula.
\end{proof}


In particular, as a direct consequence of this theorem, we have the following corollary in
3-dimensional simply connected space forms: the Euclidean space $\R{3}$, the hyperbolic space
$\h{3}(c)$, and the standard sphere $\s{3}(c)$; for this last case the result was achieved before
(see \cite{ABBr}).

\begin{corollary}
\label{th1.2} Let $\m^2$ be a compact two-sided surface with constant mean curvature $H$ immersed
into a 3-dimensional simply connected space form $\n^3(c)$, and let $\lambda_1$ stand for the first
eigenvalue of its Jacobi operator. Then
\begin{itemize}
\item[(i)] either $\lambda_1=-2(H^2+c)$ (and \m\ is totally umbilic in $\n^3(c)$), or
\item[(ii)] $\lambda_1\leq -4(H^2+c)$, with equality if and only if $\m^2$ is a Clifford torus in $\s{3}(c)$.
\end{itemize}
\end{corollary}
\begin{proof}
Since $\Kbar\equiv c$ we have $\RicN=2c$, so the normal direction of $\m$ is a direction of minimum
Ricci curvature of $\n^{3}(c)$. Now, if $\m$ is totally umbilic we know from Theorem \ref{th1.1}
that $\lambda_1=-2(H^2+c)$. Otherwise, by using the fact that the genus of a constant mean
curvature non totally umbilic surface in $\n^3(c)$ is greater than or equal to 1 we obtain
$\lambda_1\leq -4(H^2+c)$. Moreover, equality holds if and only if $g=1$ and $\m$ has constant
Gaussian curvature, so by the Gauss-Bonnet formula it must be $K=0$. This occurs only when $\m^2$
is a Clifford torus in $\s{3}(c)$.
\end{proof}

We can also derive the following consequence for strongly stable CMC surfaces.

\begin{corollary}
\label{th1.3} Let $\n^3$ be a 3-dimensional Riemannian space with sectional curvature $\Kbar$
bounded from below by a constant c.
\begin{itemize}
\item[(i)] There exists no strongly stable CMC surface with $H^2+c>0$.
\item[(ii)] If $\m^2$ is a strongly stable CMC surface and $H^2+c=0$ (that is, $c=0$ and $H=0$ or $c<0$ and $H^2=-c$),
then $\m^2$ is topologically either a sphere or a torus.
\item[(iii)] If $\m^2$ is a strongly stable CMC surface and $H^2+c<0$ (that is, $c<0$ and $H^2<-c$),
then
\[
\area|H^2+c|\geq 2\pi(g-1).
\]
\end{itemize}
\end{corollary}
The proof of item (i) above is a direct application of the estimate for $\lambda_1$ given in item
(i) of Theorem \ref{th1.1}, while items (ii) and (iii) follow from the estimate for $\lambda_1$
given in item (ii) of Theorem \ref{th1.1}.

\section{Surfaces in homogeneous 3-manifolds}
\label{s2}



From now on, we will focus our attention on the study of compact CMC surfaces
into homogeneous Riemannian 3-manifolds whose isometry group has dimension
4. So, if $\n^3$ is a homogeneous Riemannian 3-manifold, it is well known that there exists a Riemannian submersion $\Pi:\n^3\fle B^2(\kappa)$, where $B^2(\kappa)$ is a 2-dimensional simply connected space form of constant curvature $\kappa$, with totally geodesic fibers and there exists a unit Killing field $\xi$ on $\n^3$ which is vertical respect to $\Pi$. If  $\nablabar$ stands for the Levi-Civita connection of $\n^3$, we have
\begin{equation}
\label{tau}
\nablabar_X\xi=\tau(X\wedge\xi),
\end{equation}
for all vector fields $X$ on $\n^3$, where
$\wedge$ is the vector product in $\n^3$ and the constant $\tau$ is the bundle curvature (see  \cite{Dan} for details). As the
isometry group of $\n^3$ has dimension 4, $\kappa-4\tau^2\neq0$. We will denote such manifolds and certain quotients by $\E3$.
Depending on $\tau$ and $\kappa$, we can distinguish the different cases:

\begin{enumerate}
\item[(i)] When $\tau=0$, $\E3$ is the product $B^2(\kappa)\times\R{}$, that is, up to scaling, the
spaces $\s{2}(\kappa)\times\R{}$ for $\kappa>0$, and $\h{2}(\kappa)\times\R{}$ for $\kappa<0$, and their quotients $B^2(\kappa)\times\s{1}$.

\item[(ii)] When $\tau\neq0$, $\E3$ is one of the Berger spheres $\sB$ for $\kappa>0$, the Heisenberg group $\nil$ for $\kappa=0$ and the universal cover $\univC$ of the Lie group $\sL$ for $\kappa<0$, and their quotients $\sB/\Z_n$ and $PSl(2,\R{})=\sL/\Z_n$, $n\geq 2$.
\end{enumerate}


The submersion $\Pi:\E3\fle B^2(\kappa)$ allows us to get a very interesting example of surface. So, if $\gamma$ is any regular curve in $B^2(\kappa)$ we know that $\m=\Pi^{-1}(\gamma)$ is a surface in
$\E3$ with $\xi$ a tangent vector field, i.e., $\g{N}{\xi}\equiv0$. From \rf{tau} it follows that $\xi$ is
a parallel vector field on $\m$ and hence $\m$ is a flat surface. We will call $\Pi^{-1}(\gamma)$
a Hopf surface of $\E3$. If $\gamma$ is a closed curve, $\Pi^{-1}(\gamma)$ is a Hopf cylinder, and additionally if  $\Pi$ is a circle Riemannian submersion, $\Pi^{-1}(\gamma)$ is a Hopf torus. Let us observe that a Hopf torus $\Pi^{-1}(\gamma)$ has constant mean curvature if and only if the curve $\gamma$ has constant curvature.

In this section we start the study of the first stability eigenvalue of any compact CMC surface $\xhom$. According to \rf{lambda1}, we have to consider the curvature tensor of any $\E3$. It is well known (see \cite{Dan}) that the Riemannian curvature tensor $\RiemE$ of $\E3$ is given by
\begin{eqnarray*}
\g{\RiemE(X,Y)Z}{W}&=&(\kappa-3\tau^2)\{\g{Y}{Z}\g{X}{W}-\g{X}{Z}\g{Y}{W}\}+\\
&+&(\kappa-4\tau^2)\{\g{X}{\xi}\g{Z}{\xi}\g{Y}{W}-\g{Y}{\xi}\g{Z}{\xi}\g{X}{W}+\\
&+&\g{X}{Z}\g{Y}{\xi}\g{\xi}{W}-\g{Y}{Z}\g{X}{\xi}\g{\xi}{W}\},
\end{eqnarray*}
for all vector fields $X$, $Y$, $Z$ and $W$ on $\E3$.
As a consequence, the Ricci curvature of $\E3$ is given by
\begin{equation}
\label{Ric}
\Ric(X,X)=\kappa-2\tau^2+\g{X}{\xi}^2(4\tau^2-\kappa),
\end{equation}
for every unit vector field $X$ on $\E3$. Moreover, the sectional curvature $\Kbar$ of any plane
$P$ is
\[
\Kbar(P)=\tau^2+\g{\nu}{\xi}^2(\kappa-4\tau^2),
\]
where $\nu$ is the normal to $P$.

A direct computation using the last two formulae shows that
\[
\label{2K+Ric} 2\KbarSigma+\RicN=\kappa+\g{N}{\xi}^2(\kappa-4\tau^2).
\]
Therefore, equation \rf{lambda1} reduces to
\begin{equation}
\label{lambda1Hom}
\lambda_1=-4H^2-\kappa-\frac{1}{\area}\left(\alpha+8\pi(g-1)+(\kappa-4\tau^2)\integ\g{N}{\xi}^2\right).
\end{equation}

Now, as a direct application of the above equation we can get upper
bounds for $\lambda_1$ for compact CMC surfaces into the different homogeneous spaces $\E3$.
 Moreover, in some cases we characterize the equality.

\begin{theorem}
\label{ThS2R}
Let $\psi:\m^2\fle\s{2}(\kappa)\times\R{}$ be a compact two-sided surface of constant mean
curvature $H$, and let $\lambda_1$ stand for the first eigenvalue of its Jacobi operator. Then
\begin{itemize}
\item[(i)] $\lambda_1\leq-2H^2$, with equality if and only if $\m^2$ is a horizontal slice
$\s{2}(\kappa)\times\{t\}$;
\item[(ii)] $\displaystyle{
\lambda_1<-4H^2-\kappa-\frac{8\pi(g-1)}{\area}}.$
\end{itemize}
\end{theorem}
\begin{proof} Since $\tau=0$ and $\kappa>0$ from \rf{Ric} we know that
\[
\RicN=\kappa(1-\g{N}{\xi}^2)\geq0,
\]
so that \rf{lambda1.1} directly yields $\lambda_1\leq-2H^2$. Moreover, if the equality holds then
$\g{N}{\xi}^2\equiv1$ which implies that $\m^2$ is a totally geodesic horizontal slice.

On the other hand, since $\tau=0$ from \rf{lambda1Hom}, we get
\begin{eqnarray*}
\lambda_1&=&-4H^2-\kappa-\frac{1}{\area}\left(\alpha+8\pi(g-1)+\kappa\integ\g{N}{\xi}^2\right)\\
&\leq&-4H^2-\kappa-\frac{8\pi(g-1)}{\area},
\end{eqnarray*}
where we are taking into account that $\kappa>0$ and $\g{N}{\xi}^2\geq0$. In fact, equality holds
if and only if $\alpha=0$ and $\g{N}{\xi}\equiv0$. This last condition implies that $\m^2$ would be
a Hopf torus but that is not possible because there are no flat compact surfaces in
$\s{2}(\kappa)\times\R{}$ (see \cite{ToUr}).
\end{proof}
As a consequence, let us note that in $\s{2}(\kappa)\times\R{}$ if $H\neq0$ then $\lambda_1<0$ (so
the surface is not strongly stable) and from here we easily get that the only constant mean
curvature compact strongly stable surfaces in $\s{2}(\kappa)\times\R{}$ are horizontal slices
$\s{2}(\kappa)\times\{t\}$.
\begin{corollary}
The only strongly stable compact surfaces of constant mean curvature in $\s{2}(\kappa)\times\R{}$
are horizontal slices.
\end{corollary}

\begin{theorem}
\label{ThS2S1}
Let $\psi:\m^2\fle\s{2}(\kappa)\times\s{1}$ be a compact two-sided surface of constant mean
curvature $H$, and let $\lambda_1$ stand for the first eigenvalue of its Jacobi operator. Then
\begin{itemize}
\item[(i)] $\lambda_1\leq-2H^2$, with equality if and only if $\m^2$ is a horizontal slice
$\s{2}(\kappa)\times\{p\}$;
\item[(ii)] $\displaystyle{
\lambda_1\leq-4H^2-\kappa-\frac{8\pi(g-1)}{\area}}$
and equality holds if and only if $\m^2$ is a Hopf torus over a constant curvature closed curve.
\end{itemize}
\end{theorem}
\begin{proof}
The same as above but in this case for the equality in item (ii) we do have CMC Hopf tori.
Let us observe that from \rf{lambda1Hom} the first eigenvalue for such tori is
$\lambda_1=-4H^2-\kappa$ because of $g=1$, $\g{N}{\xi}\equiv0$ and $\alpha=0$.
\end{proof}

\begin{theorem}
\label{Th H2R}
Let $\psi:\m^2\fle\h{2}(\kappa)\times\R{}$ be a compact two-sided surface of constant mean
curvature $H$, and let $\lambda_1$ stand for the first eigenvalue of its Jacobi operator. Then
\begin{itemize}
\item[(i)] $\lambda_1<-2H^2-\kappa$;
\item[(ii)] $\displaystyle{
\lambda_1< -4H^2-2\kappa-\frac{8\pi(g-1)}{\area}}$.
\end{itemize}
\end{theorem}
\begin{proof} In this case, since $\tau=0$ and $\kappa<0$ from \rf{Ric} we know that
\[
\RicN=\kappa(1-\g{N}{\xi}^2)\geq\kappa,
\]
so that \rf{lambda1.1} directly yields $\lambda_1\leq-2H^2-\kappa$. Moreover, the equality cannot hold; otherwise
we would have $\g{N}{\xi}^2\equiv0$ which implies that $\m^2$ is a cylinder which is not possible because it is compact.

On the other hand, since $\tau=0$, $\kappa<0$ and $\g{N}{\xi}^2\leq1$ we have from
\rf{lambda1Hom}
\[
\lambda_1\leq-4H^2-2\kappa-\frac{8\pi(g-1)}{\area},
\]
and equality holds if and only $\alpha=0$ and $\g{N}{\xi}^2\equiv1$. So $N=\pm\xi$ which implies
that the surface would be a slice of $\h{2}(\kappa)\times\R{}$ that is not compact.
\end{proof}

As a consequence, we derive the following result which is related to Corollary 4.1 in \cite{NR} (it is important to realize that
the notion of stability in Corollary 4.1 of \cite{NR}, as well as in Theorem C of the same reference, is that of weakly stability
although not explicitly stated).
\begin{corollary}
Let $\psi:\m^2\fle\h{2}(\kappa)\times\R{}$ be a compact two-sided surface of constant mean
curvature $H$.
\begin{itemize}
\item[(i)] There exists no strongly stable CMC surface with $H^2\geq-\kappa/2$.
\item[(ii)] If $\m^2$ is a strongly stable CMC surface and $H^2<-\kappa/2$,
then
\[
\area|2H^2+\kappa|> 4\pi(g-1).
\]
\end{itemize}
\end{corollary}

We study now the cases where the bundle curvature $\tau\neq0$.

\begin{remark}
\label{rem1} When this happens, we know that $\{p\in\m^2:\g{N}{\xi}^2(p)=1\}=\{p\in\m^2:\xi(p)=\pm
N(p)\}$ has empty interior because the distribution $\langle\xi\rangle^\bot$ on $\E3$ is not
integrable (see \cite{ToUr}).
\end{remark}

We begin with the most simple case, that is, the Heisenberg group $\nil$ where $\kappa=0$.
\begin{theorem}
\label{Nil3}
Let $\psi:\m^2\fle\nil$ be a compact two-sided surface of constant mean curvature $H$, and let $\lambda_1$
stand for the first eigenvalue of its Jacobi operator. Then
\begin{itemize}
\item[(i)] $\lambda_1<-2(H^2-\tau^2)$;
\item[(ii)]$\displaystyle{
\lambda_1< -4(H^2-\tau^2)-\frac{8\pi(g-1)}{\area}}.$
\end{itemize}
\end{theorem}
\begin{proof} Since $\kappa=0$, from \rf{Ric} we know that
\[
\RicN=2\tau^2(2\g{N}{\xi}^2-1)\geq-2\tau^2,
\]
so that \rf{lambda1.1} directly yields $\lambda_1\leq-2H^2+2\tau^2$. Moreover, the equality cannot hold because by the results in
\cite{SoTo} we know that there is no totally umbilic surfaces in \nil.

On the other hand, by \rf{lambda1Hom} we get
\begin{eqnarray*}
\lambda_1&=&-4H^2-\frac{1}{\area}\left(\alpha+8\pi(g-1)-4\tau^2\integ\g{N}{\xi}^2\right)\\
&\leq&-4H^2+4\tau^2-\frac{8\pi(g-1)}{\area}.
\end{eqnarray*}
If the equality holds then $\g{N}{\xi}\equiv 1$ and it is not possible because of Remark
\ref{rem1}.
\end{proof}

\begin{corollary}
Let $\psi:\m^2\fle\nil$ be a compact two-sided surface of constant mean
curvature $H$.
\begin{itemize}
\item[(i)] There exists no strongly stable CMC surface with $H^2\geq\tau^2$.
\item[(ii)] If $\m^2$ is a strongly stable CMC surface and $H^2<\tau^2$,
then
\[
\area|H^2-\tau^2|> 2\pi(g-1).
\]
\end{itemize}
\end{corollary}

Now, let $\kappa$ be positive, thus,  $\E3=\sB$ is a Berger sphere.
For these homogeneous Riemannian manifold is very common to consider two different cases
based on the sign of $\kappa-4\tau^2$ since the obtained results are quite different.


\begin{theorem}
\label{thsb} Let $\psi:\m^2\fle\sB$ be a compact two-sided surface of constant mean curvature $H$,
and let $\lambda_1$ stand for the first eigenvalue of its Jacobi operator.
\begin{enumerate}
\item[(a)] If $\kappa-4\tau^2>0$ then

\vspace{2mm}
\begin{itemize}
\item[(i)] $\lambda_1<-2(H^2+\tau^2)$,
\item[(ii)] $\displaystyle{\lambda_1\leq -4H^2-\kappa-\frac{8\pi(g-1)}{\area}}$ and
equality holds if and only if $\m^2$ is a Hopf torus over a constant curvature closed curve.
\end{itemize}
\vspace{2mm}
\item[(b)] If $\kappa-4\tau^2<0$ then
\vspace{2mm}
\begin{itemize}
\item[(i)] $\lambda_1<-2H^2-\kappa+2\tau^2$,
\item[(ii)] $\displaystyle{\lambda_1< -4H^2-2\kappa+4\tau^2-\frac{8\pi(g-1)}{\area}}$.
\end{itemize}
\end{enumerate}
\end{theorem}
\begin{proof}(a) The proof of item (i) uses the fact that in this case
\[
\RicN\geq2\tau^2,
\]
which by \rf{lambda1.1} gives $\lambda_1\leq-2(H^2+\tau^2)$. Besides the equality cannot hold because of the non-existence of totally umbilic
surfaces in the Berger spheres \cite{SoTo}. On the other hand, since $\kappa-4\tau^2>0$, by \rf{lambda1Hom} we get
\[
\lambda_1\leq -4H^2-\kappa-\frac{8\pi(g-1)}{\area}.
\]
If the equality holds then $\m^2$ has to be a Hopf torus because of $\g{N}{\xi}\equiv0$ and reciprocally, any CMC Hopf torus satisfies the equality as we have seen before.

(b) In this case,
\[
\RicN\geq\kappa-2\tau^2,
\]
and \rf{lambda1.1} yields $\lambda_1\leq-2(H^2-\tau^2)-\kappa$. Again, equality cannot happen since there exist no totally umbilic surfaces.
Since $\kappa-4\tau^2<0$, by \rf{lambda1Hom} we have
\[
\lambda_1\leq -4H^2-2\kappa+4\tau^2-\frac{8\pi(g-1)}{\area},
\]
and equality holds if and only $\alpha=0$ and $\g{N}{\xi}^2\equiv1$ but again it is not possible.
\end{proof}

The corresponding applications to strongly stable surfaces in the Berger spheres have no sense because we know by
\cite{MPR} that there exist no such surfaces in these ambient spaces. Finally, for the case $\kappa<0$ we have the following result.

\begin{theorem}
Let $\psi:\m^2\fle\univC$ be a compact two-sided surface of constant mean curvature $H$, and let
$\lambda_1$ stand for the first eigenvalue of its Jacobi operator. Then
\begin{itemize}
\item[(i)] $\lambda_1<-2H^2-\kappa+2\tau^2$,
\item[(ii)] $\displaystyle{\lambda_1< -4H^2-2\kappa+4\tau^2-\frac{8\pi(g-1)}{\area}}$.
\end{itemize}
\end{theorem}
\begin{proof} Since $\kappa<0$ we get $\kappa-4\tau^2<0$, so the proof is the same that the part (b)
of the above theorem.
\end{proof}

\begin{corollary}
Let $\psi:\m^2\fle\univC$ be a compact two-sided surface of constant mean
curvature $H$.
\begin{itemize}
\item[(i)] There exists no strongly stable CMC surface with $H^2\geq\tau^2-\kappa/2$.
\item[(ii)] If $\m^2$ is a strongly stable CMC surface and $H^2<\tau^2-\kappa/2$,
then
\[
\area|H^2-\tau^2+\kappa/2|> 2\pi(g-1).
\]
\end{itemize}
\end{corollary}

\bibliographystyle{amsplain}

\begin{thebibliography}{10}

\bibitem{AR} U. Abresch and H. Rosenberg, \textit{A Hopf differential for constant mean curvature surfaces in $\s{2}\times\R{}$ and $\h{2}\times\R{}$}, Acta Math. {\bf 193} (2004), 141--174.

\bibitem{ABBr} L. J. Al\'{\i}as, A. Barros and A. Brasil Jr., \textit{A spectral characterization of the $H(r)$-torus by the first stability eigenvalue}, Proc. Amer. Math. Soc. \textbf{133} (2005),  no. 3, 875--884.

\bibitem{Dan} B. Daniel, \textit{Isometric immersions into 3-dimensional homogeneous manifolds},
Comment. Math. Helv. {\bf 82} (2007), 87--131.

\bibitem{MPR} W. H. Meeks, J. P\'erez and A. Ros, \textit{Stable constant mean curvature surfaces.  Handbook of geometric analysis}, no. 1,  301--380, Adv. Lect. Math. (ALM), 7, Int. Press, Somerville, MA, (2008).

\bibitem{NR} B. Nelli and H. Rosenberg, \textit{Global properties of constant mean curvature surfaces in $\h{2}\times\R{}$}, Pacific J. Math. {\bf 226} (2006), no. 1, 137--152.

\bibitem{Pe} O. Perdomo, \textit{First stability eigenvalue characterization of Clifford hypersurfaces},
Proc. Amer. Math. Soc. {\bf 130} (2002), 3379--3384.

\bibitem{Sim} J. Simons, \textit{Minimal varieties in Riemannian manifolds}, Ann.
of Math. (2) {\bf 88} (1968), 62--105.

\bibitem{S} R. Souam, \textit{On stable constant mean curvature surfaces in $\s{2}\times\R{}$ and $\h{2}\times\R{}$}, Trans. Amer. Math. Soc. {\bf 362} (2010), no. 6, 2845--2857.

\bibitem{SoTo} R. Souam and E. Toubiana, \textit{Totally umbilic surfaces in homogeneous 3-manifolds}, Comment. Math. Helv. {\bf 84} (2009),  no. 3, 673--704.

\bibitem{ToUr} F. Torralbo and F. Urbano, \textit{On the Gauss curvature of compact surfaces in
homogeneous 3-manifolds}, Proc. Amer. Math. Soc. {\bf 138} (2010), 2561--2567.

\bibitem{ToUrStability} F. Torralbo and F. Urbano, \textit{Compact stable constant mean curvature surfaces in
homogeneous 3-manifolds}, to appear in Indiana U. Math. J.

\bibitem{Wu} C. Wu, \textit{New characterizations of the Clifford tori and the
Veronese surface}, Arch. Math. (Basel) {\bf 61} (1993), 277--284.

\end{thebibliography}

\end{document}